\documentclass[11pt,oneside,reqno]{amsart}
\usepackage{amsmath,amssymb,amsbsy,amsfonts,amsthm,
latexsym,amsopn,amstext, amsxtra,euscript,amscd,hyperref}
\usepackage{xcolor}
\newtheorem{theorem}{Theorem}[section]
\newtheorem{definition}{Definition}[section]
\newtheorem{prop}{Proposition}[section]
\newtheorem{corollary}{Corollary}[section]
\newtheorem{lemma}{Lemma}[section]

\def \N{\mathbb{N}}

\def \Z{\mathbb{Z}}

\begin{document}

\title[]{On the rational approximation to linear combinations of powers}
\author{Veekesh Kumar and Gorekh Prasad}

\pagenumbering{arabic}

\begin{abstract} 
For a complex number $x$, $\Vert x\Vert:=\min\{|x-m|:m\in\mathbb{Z}\}$. 
Let $k\geq 1$ be an integer, and $K$ be a number field. Let $\alpha_1,\ldots,\alpha_k$ be  algebraic numbers with $|\alpha_i|\geq 1$ and let $d_i$ denotes the degree of $\alpha_i$ for $1\leq i\leq k$.  Set $d=d_1+\cdots+d_k$. 
In this article, we show  that if  the inequality 
$
0<\Vert\lambda_1 q\alpha^n_1+\cdots+\lambda_k q\alpha^n_k\Vert<\frac{\theta^n}{q^{d+\varepsilon}}
$
has infinitely many solutions in $(n, q,\lambda_1,\ldots,\lambda_k)\in \mathbb{N}^2\times (K^\times)^k$ with absolute logarithmic Weil height of $\lambda_i$ is small compared to $n$ and some $\theta\in (0,1)$,  then, in particular,  the tuple $(\lambda_1 q\alpha^n_1,\ldots, \lambda_k q\alpha^n_k)$ is pseudo-Pisot, and at least one of $\alpha_i$ is an algebraic integer. This result can be viewed as Roth's type theorem for linear combinations of powers of algebraic numbers over $\overline{\mathbb{Q}}$. The case  $q=1$ was recently proved by Kulkarni, Mavraki, and Nguyen \cite{kul}, which is a generalization of Mahler's question proved in \cite{corv}. As a consequence of our result,  we obtain the following generalization of this question: let $\alpha>1$ be an algebraic number with $d=[\mathbb{Q}(\alpha):\mathbb{Q}]$. For a given $\varepsilon>0$, if  the inequality 
$$
0<\Vert\lambda q\alpha^n\Vert<\frac{\theta^n}{q^{d+\varepsilon}}
$$
has infinitely many solutions in the tuples $(n,q,\lambda)\in \mathbb{N}^2\times K^\times$ with absolute logarithmic Weil height of $\lambda$ is small compared to $n$ and $\theta\in (0,1)$, then some power of $\alpha$ is  a Pisot number. As an application of this result, we deduce the transcendence of certain infinite products of algebraic numbers.
\end{abstract}
\address[Veekesh Kumar]{Department of Mathematics, Indian Institute of Technology, Dharwad  580011, Karnataka, India.}

\email[]{veekeshk@iitdh.ac.in}
\address[Gorekh Prasad]{Harish-Chandra Research Institute, A CI of Homi Bhaba National Institute, Prayagraj, 211019, India.}

\email[]{gorekhprasad@hri.res.in}
\subjclass[2010] {Primary 11J68, 11J87;  Secondary 11B37, 11R06 }
\keywords{Rational approximation, Linear recurrence sequence, Pisot number,  Schmidt Subspace Theorem}

\maketitle

\section{\bf{Introduction}}

\medskip
For a complex number $x$,  $\Vert x\Vert$  denotes the distance of $x$ from its nearest integer in $\mathbb{Z}$. In other words, 
$$
\Vert x\Vert:=\mbox{min}\{|x-m|:m\in\mathbb{Z}\}.
$$
Mahler \cite{mahler} in 1957 showed that for   $\alpha\in \mathbb{Q}\backslash\Z$ with $\alpha>1$ and any real number $\varepsilon>0$, there are only finitely many $n\in \N$ satisfying  $\Vert\alpha^n\Vert<2^{-\varepsilon n}$. As a consequence of Mahler's result,  the number $g(k)$ in Waring's problem satisfies: 
$$
g(k)=2^k+\left[\left(\frac{3}{2}\right)^k\right]-2
$$
except for finitely many values of $k$. In the same paper, Mahler asked: Classify all the algebraic numbers as having the same property as the non-integral rationals.
\vspace{.2cm}

In 2004, by ingenious applications of the Subspace Theorem,  Corvaja and Zannier \cite{corv} proved the following  {\it Thue-Roth type inequality} with {\it moving targets} to answer Mahler's question.

\begin{theorem}\label{maintheorem}{\rm(Corvaja and Zannier)}~
Let $\alpha, \lambda\neq 0$ be  algebraic numbers  and $d=[\mathbb{Q}(\alpha):\mathbb{Q}]$. Let $\varepsilon$ be a positive real number. Suppose the set $\mathcal{A}$ of pairs $(n,q)$ satisfying the inequalities $|\lambda q\alpha^n|>1$ and $0<\Vert\lambda q\alpha^n \Vert< \frac{1}{H(\alpha^n)^\varepsilon q^{d+\varepsilon}}$ is infinite, where $H(\alpha)$ is the absolute Weil height of an algebraic number. Then $\lambda q\alpha^n$ is pseudo-Pisot for all but finitely many pairs $(n,q)\in\mathcal{A}$. 
\end{theorem} 

As introduced in \cite{corv}, an algebraic number $\alpha$ is said to be a {\it pseudo-Pisot number} if $|\alpha|>1$ and all its other Galois conjugates over $\mathbb{Q}$ have an absolute value less than one, and $\alpha$ has an integral trace. pseudo-Pisot numbers which are algebraic integers are precisely the Pisot numbers.
\smallskip

As a consequence of Theorem \ref{maintheorem}, they settled the question of Mahler and proved the following result: let $\alpha > 1$ be an algebraic number. Suppose for some $\theta \in (0,1)$, the inequality $\Vert\alpha^n \Vert < \theta^n$ has infinitely many solutions for $n \in \mathbb{N}$. Then, there exists an integer $m \geq 1$ such that $\alpha^m$ is a Pisot number. This is equivalent to saying that if the inequalities in Theorem \ref{maintheorem} have infinitely many solutions in pairs $(n,1)\in\mathcal{A}$,  and $\lambda=1$, then the conclusion of Theorem \ref{maintheorem} is that some power of $\alpha$ is a Pisot number.
\smallskip

Therefore, it is natural to ask the following question:
\bigskip

\noindent{\bf Question 1.}
Let $\alpha>1$ be an algebraic number with $d=[\mathbb{Q}(\alpha):\mathbb{Q}]$. Let $K$ be a number field. Suppose for some $\varepsilon>0$  and $\theta\in (0,1)$ the inequality 
$$
0<\Vert\lambda q\alpha^n\Vert<\frac{\theta^n}{q^{d+\varepsilon}}
$$
has infinitely many solutions in the triples $(n,q,\lambda)\in \mathbb{N}^2\times K^\times$ with absolute logarithmic Weil height of $\lambda$ is small compared to $n$. Then what can be said about $\alpha$? Can we say that some power of $\alpha$ is a Pisot number?

\bigskip

 Before stating further  results, we recall the following definitions introduced in \cite{kul}.
\begin{definition}
Let $(\beta_1,\ldots,\beta_k)$ be  a tuple of distinct non-zero algebraic numbers. Set 
$$
B:=\{\beta\in\bar{\mathbb{Q}}^\times\backslash\{\beta_1,\ldots,\beta_k\}:~\beta=\sigma(\beta_i)~~\mbox{for some~~}\sigma:\mathbb{Q}(\beta_1,\ldots,\beta_k)\to \mathbb{C}~~\mbox{and~~} 1\leq i\leq k\}.
$$
 Then the tuple $(\beta_1,\ldots,\beta_k)$ is  called pseudo-Pisot if $\sum_{i=1}^k\beta_i+\sum_{\beta\in B}\beta\in\mathbb{Z}$  and $|\beta|<1$  for every $\beta\in B$. Moreover, if $\beta_i$ is an algebraic integer for $1\leq i\leq k$ then the tuple $(\beta_1,\ldots,\beta_k)$ is called Pisot.
\end{definition}
\begin{definition}
 A tuple $(\alpha_1, \ldots, \alpha_k)$ of non-zero algebraic numbers is {\it non-degenerate} if $\alpha_i/\alpha_j$ is not a root of unity  for  all integers $1\leq i < j \leq k$.
\end{definition}
Let $h(x)$ denote the absolute logarithmic Weil height, see Section 2 below.   
A  function $f:\mathbb{N}\rightarrow (0,\infty)$ satisfying $\displaystyle\lim_{n\to\infty}\frac{f(n)}{n}=0$ is called a sublinear  function.  Let $G_\mathbb{Q}$ be the absolute Galois group of $\mathbb{Q}.$  
\bigskip

With these  notations, we  state the recent result of Kulkarni, Mavraki and Nguyen proved in  \cite{kul}, which generalizes  Mahler's question. More precisely, the following theorem.
\begin{theorem}\label{maintheorem1}{\rm(Kulkarni, Mavraki and Nguyen)}
Let $k\in\mathbb{N}$ and let $(\alpha_1,\ldots,\alpha_k)$ be a  non-degenerate tuple of   algebraic numbers with $|\alpha_i|\geq 1$ for $1\leq i\leq k$.  Let $K$ be a number field and $f$ be a sublinear function.  Suppose that for some $\theta\in (0,1)$, the set $\mathcal{A}$ of  tuples $(n,\lambda_1,\ldots,\lambda_k)\in\mathbb{N}\times (K^\times)^k$ satisfying the inequality
\begin{equation*}\label{eq1.1}
\tag{1.1}
\Vert\lambda_1 \alpha^n_1+\cdots+\lambda_k \alpha^n_k\Vert<\theta^n\quad \mbox{and}~~~ \max_{1\leq i\leq k}h(\lambda_i)<f(n)
\end{equation*}
is  infinite. Then 
\begin{enumerate}
\item[(i)] Each  $\alpha_i$  is an   algebraic integer. 
\item[(ii)] For each $\sigma\in G_\mathbb{Q}$  and $1\leq i\leq k$ such that $\frac{\sigma(\alpha_i)}{\alpha_j}$ is not a root of unity for $1\leq j\leq k$, we have $|\sigma(\alpha_i)|<1$.
\end{enumerate}
Moreover, for all but finitely many  tuples $(n,\lambda_1,\ldots,\lambda_k)\in\mathcal{A}$, the following hold:
\begin{enumerate}
\item[(iii)]
$\sigma(\lambda_i \alpha^n_i)=\lambda_j\alpha^n_j$ precisely for  those triples $(\sigma, i, j)\in G_\mathbb{Q}\times\{1,\ldots,k\}^2$ such that $\frac{\sigma(\alpha_i)}{\alpha_j}$ is a root of unity.
\item[(iv)]
 The tuple $(\lambda_1  \alpha^n_1,\ldots,\lambda_k  \alpha^n_k)$ is pseudo-Pisot. 
\end{enumerate}
\end{theorem}

By setting $k=1$ and $\lambda_1=1$, properties $(i)$ and $(iv)$ demonstrate that $\alpha^m$ is a Pisot number, which answers Mahler's question. Also, by using full strength of this theorem,  we can partially address Question 1. Specifically, when $h(\lambda q)<f(n)$, we can establish that some power of $\alpha$ is a Pisot number.   The main purpose of this article is to strengthen Theorem \ref{maintheorem} and  extend Theorem \ref{maintheorem1} to a more general sequence of the form $\{q\lambda_1\alpha^n_1+\cdots+q\lambda_k\alpha^n_k: n\in\mathbb{N}\}$, where $\lambda_i$'s are defined as in Theorem \ref{maintheorem1}, and $q\in\mathbb{N}$. As a result of this extension, we provide an answer to Question 1 when absolute logarithmic Weil height of $\lambda$ is small compared to $n$.
\smallskip
%

Here,  we have our first  theorem.
\begin{theorem}\label{maintheorem2}
Let $k\in\mathbb{N}$ and let $(\alpha_1,\ldots,\alpha_k)$ be a  non-degenerate tuple of   algebraic numbers. Let $d_i=[\mathbb{Q}(\alpha_i):\mathbb{Q}]$ for $1\leq i\leq k$ and set $d=d_1+\cdots+d_k$.  Let $K$ be a number field and $f$ be a sublinear function. Let $\varepsilon$ be a positive real number. Suppose for some $\theta\in (0,1)$, the set $\mathcal{A}$ of  tuples $(n,q, \lambda_1,\ldots,\lambda_k)\in\mathbb{N}^2\times (K^\times)^k$ such that $|q\alpha^n_i|\geq 1$ for $i=1,\ldots,k$ and  satisfying the inequalities 
\begin{equation}\label{eq1.2}
\tag{1.2}
0<\Vert\lambda_1 q\alpha^n_1+\cdots+\lambda_k q\alpha^n_k\Vert<\frac{\theta^n}{q^{d+\varepsilon}}\quad \mbox{and}~~~ \max_{1\leq i\leq k}h(\lambda_i)<f(n)
\end{equation}
is  infinite. Then 
\begin{enumerate}
\item[(i)] At least one of the $\alpha_i$'s is an algebraic integer and if  $h(q\lambda_i)<f(n)$ for $i=1,\ldots,k$ and for infinitely many $(n,q,\lambda_1,\ldots,\lambda_k)\in\mathcal{A}$,   then  all the  $\alpha_i$'s  are  algebraic integers. 
\item[(ii)] For each $\sigma\in G_\mathbb{Q}$  and $1\leq i\leq k$ such that $\frac{\sigma(\alpha_i)}{\alpha_j}$ is not a root of unity for $1\leq j\leq k$, we have $|\sigma(\alpha_i)|<1$.
\end{enumerate}
Moreover, for all but finitely many  tuples $(n,q,\lambda_1,\ldots,\lambda_k)\in\mathcal{A}$, the following hold:
\begin{enumerate}
\item[(iii)]
$\sigma(\lambda_i \alpha^n_i)=\lambda_j\alpha^n_j$ precisely for  those triples $(\sigma, i, j)\in G_\mathbb{Q}\times\{1,\ldots,k\}^2$ such that $\frac{\sigma(\alpha_i)}{\alpha_j}$ is a root of unity.
\item[(iv)]
 The tuple $(\lambda_1 q \alpha^n_1,\ldots,\lambda_k q \alpha^n_k)$ is pseudo-Pisot and no subsum of $\sum_{i=1}^k\lambda_i q\alpha^n_i$  vanishes under the trace map.
\end{enumerate}
\end{theorem}
In the case where $h(q\lambda_i) < f(n)$, Theorem \ref{maintheorem2} coincides with Theorem \ref{maintheorem1}. To see this, we set $\lambda_1' = q\lambda_1$, $\ldots,$ $\lambda'_k = q\lambda_k$. Then, we have $0 < \Vert\lambda'_1\alpha^n_1 + \cdots + \lambda'_k\alpha^n_k\Vert < \theta^n$ with $h(\lambda'_i) < f(n)$ and for some $\theta \in (0,1)$. This is equivalent to the inequality \eqref{eq1.1} in Theorem \ref{maintheorem1}. Furthermore, it also explains that, without loss of generality, one can take $q = 1$ when $h(q\lambda_i) < f(n)$.
\bigskip

As a consequence of Theorem \ref{maintheorem2}, when absolute logarithm Weil height of $\lambda$ is small compared to $n$, we settle Question 1 in the following corollary.
\begin{corollary}\label{cor}
Let $\alpha>1$ be an algebraic number and $d=[\mathbb{Q}(\alpha):\mathbb{Q}]$. Let $K$ be a number field and $f$ be a sublinear function. Suppose for some $\varepsilon>0$ and $\theta\in (0,1)$, the  inequality  
$$
0<\Vert\lambda q \alpha^n\Vert<\frac{\theta^n}{q^{d+\varepsilon}}
$$
holds for infinitely many triples $(n,q, \lambda)\in\mathbb{N}^2\times K^\times$ with $h(\lambda)<f(n)$. Then 
there exists a  positive integer $m$ such that   $\alpha^m$ is a Pisot number. In particular, $\alpha$ is an algebraic integer.
\end{corollary}

It is important to note that the assumption $\Vert\lambda q\alpha^n\Vert\neq 0$ in the above corollary is impossible to avoid. For example, consider $\alpha=\frac{3}{2}$, $q=2^n$, and $\lambda\in\mathbb{Z}$. In this case, we have $\Vert\lambda q\alpha^n\Vert=0$, but $\alpha$ is not an integer.
\smallskip

The article is organized as follows: Section \ref{sec3} gives an overview of Weil heights and the technical ingredients needed to prove our results, including the famous   Schmidt's Subspace Theorem and some of its applications. In Section \ref{sec4}, we present the proofs of our results. In Section \ref{sec5}, we provide some applications of our result, where we prove some transcendence results. The approach we take up in this paper is an adaptation of the papers \cite{corv}, \cite{kul}, \cite{kumar} and \cite{kumar1}, with suitable modifications.

\section{\bf{Preliminaries}}\label{sec3} 

For any number field $K$, let $M_K$ be the set of all places on $K$ and $M_K^\infty$ be the set of all archimedean places on $K$. 
 For each place $w\in M_K$, let $K_w$ denote the completion of the number field $K$ with respect to $w$ and $d(w)=[K_w:\mathbb{Q}_\mathit{v}]$, where $\mathit{v}$ is the restriction of $w$ to $\mathbb{Q}$. 
For  every $w\in M_K$ whose restriction on $\mathbb{Q}$ is $v$ and $\alpha\in K$, we define the  normalized  absolute value $|\cdot|_w$ as follows:
\begin{equation*}
|\alpha|_w:=|\mbox{Norm}_{K_w/\mathbb{Q}_v}(\alpha)|^{\frac{1}{[K:\mathbb{Q}]}}_v.
\end{equation*} 
With these normalization,  the product formula   $\displaystyle\prod_{\omega\in M_K}|x|_\omega=1$ holds true for any  $x\in K^\times$. 
\smallskip

For all $x\in K$,  the 
absolute Weil  height  $H(x)$ is defined as 
$$
H(x):=\prod_{w\in M_K}\mbox{max}\{1,|x|_w\}, 
$$
and the absolute logarithmic height $h(x):=\log H(\alpha)$.
\bigskip

For a vector $\mathbf{x}=(x_1,\ldots,x_n)\in K^n$  and for a place $w\in M_K$, the $w$-norm for  $\mathbf{x}$ denoted by $||\mathbf{x}||_w$ is given by 
$$
||\mathbf{x}||_w:=\mbox{max}\{|x_1|_w,\ldots,|x_n|_w\}
$$
and  the projective height,  $H(\mathbf{x})$, is defined by 
$$
H(\mathbf{x}):=\prod_{w\in M_K}||\mathbf{x}||_w.
$$

The main  tool in our proofs  is the following version of Schmidt’s Subspace Theorem, which was formulated by Schlickewei and Evertse.  For a reference,  see (\cite[Chapter 7]{bomb}, \cite{evertse} and \cite[Page 16, Theorem II.2]{schmidt}).

\smallskip

\begin{theorem}\label{schli}{\rm(Subspace Theorem)}~
 Let $K$ be an algebraic number field and $m \geq 2$ an integer. Let $S$ be a finite set of places on $K$ containing all archimedean  places.  For each $v \in S$, let $L_{1,v}, \ldots, L_{m,v}$ be linearly independent linear  forms in the variables $X_1,\ldots,X_m$ with  coefficients in $K$.  For any $\varepsilon>0$, the set of solutions $\mathbf{x} \in K^m$ to the inequality 
\begin{equation*}
\prod_{v\in S}\prod_{i=1}^{m} \frac{|L_{i,v}(\mathbf{x})|_v}{|\mathbf{x}|_v} \leq \frac{1}{H(\mathbf{x})^{m+\varepsilon}}
\end{equation*}
lies in finitely many proper subspaces of $K^m$.
\end{theorem}
We need the following  proposition,  established in \cite[Proposition 2.2]{kul} for the proofs of our results.

\begin{prop}\label{prop-kul1}
Let $(\alpha_1,\ldots,\alpha_k)$  be a non-degenerate tuple of non-zero algebraic numbers, let $f$ be a sublinear function, and let $K$ be a number field.  Then there are only finitely many tuples $(n, b_1,\ldots,b_k)\in\mathbb{N}\times(K^\times)^k$ satisfying 
$$
b_1\alpha^n_1+\cdots+b_k\alpha^n_k=0\quad\mbox{and}~~~ \max_{1\leq i\leq k}h(b_i)<f(n).
$$
\end{prop}
A slight modification of Proposition 2.3 in  \cite{kul} yields the following, which  we use in the  proof of Theorem \ref{maintheorem2}.
\begin{prop}\label{prop-kul2}
Let $K$ be a  Galois extension over $\mathbb{Q}$  and $S$  be a finite set of places, containing all  the archimedean places.  Let 
$\lambda_0, \lambda_1,\ldots,\lambda_k$  be non-zero elements of $K$. 
Let  $\varepsilon>0$  be a positive real number and $\omega\in S$  be a distinguished place.  Let $\mathfrak{E}$ be an infinite set of solutions $(u_1,\ldots,u_k, b_1,\ldots,b_k)$ of the inequality 
\begin{equation*}\label{eq2.3}
\tag{2.1}
0<\left|\sum_{j=1}^k\lambda_j b_j u_j+\lambda_0\right|_\omega\leq \frac{\max\{|b_1 u_1|_\omega,\ldots,|b_k u_k|_\omega\}}{\left(\prod_{j=1}^k H(b_j)\right)^{k+2+\varepsilon}}\frac{1}{H(1, u_1,\ldots,u_k)^\varepsilon},
\end{equation*}
 where $u_j$'s are $S$-unit and $b_j\in K^\times$  for $1\leq j\leq k$.  Then there exists a non-trivial relation of the form 
 $$
 c_1 b_1 u_1+\cdots+c_k b_k u_k=0, \quad \mbox{where~~~} c_i\in K
 $$ 
holding for infinitely many elements of $\mathfrak{E}$.
\end{prop}
\begin{proof}
By applying Proposition 2.3 in \cite{kul} to the inequality \eqref{eq2.3},  we get   a non-trivial relation of the form
\begin{equation*}\label{eq2.4}
\tag{2.2}
a_0+a_1 b_1 u_1+\cdots+a_k b_k u_k=0
\end{equation*}
satisfied by infinitely many tuples $(u_1,\ldots,u_k, b_1,\ldots,b_k)\in \mathfrak{E}$. In order to finish the proof, it is enough to claim the following.
\smallskip

\noindent{\bf CLAIM.~}  There exists a non-trivial relation as \eqref{eq2.4} with $a_0=0$. 
\bigskip

Assume that $a_0\neq 0$. By rewriting the relation \eqref{eq2.4}, we obtain 
\begin{equation*}
a_0=-a_1 b_1 u_1-\cdots-a_k b_k u_k\iff 1=-\left(\frac{a_1}{a_0}\right)b_1 u_1-\cdots-\left(\frac{a_k}{a_0}\right)b_k u_k.
\end{equation*}
So, $$\lambda_0=-\lambda_0\left(\frac{a_1}{a_0}\right)b_1 u_1-\cdots-\lambda_0\left(\frac{a_k}{a_0}\right)b_k u_k.$$
Substituting this in \eqref{eq2.3}, we get  
\begin{equation*}\label{eq2.5}
\tag{2.3}
0<\left|\left(\lambda_1-\frac{\lambda_0 a_1}{a_0}\right)b_1 u_1+\cdots+\left(\lambda_k-\frac{\lambda_0 a_k}{a_0}\right)b_k u_k\right|_\omega\leq \frac{\max\{|b_1 u_1|_\omega,\ldots,|b_k u_k|_\omega\}}{\left(\prod_{j=1}^k H(b_j)\right)^{k+2+\varepsilon}}\frac{1}{H(1,u_1,\ldots,u_k)^\varepsilon}
\end{equation*}
holds for infinitely many tuples $(u_1,\ldots,u_k, b_1,\ldots,b_k)\in \mathfrak{E}$. 
We then re-apply   Proposition 2.3 in \cite{kul} to \eqref{eq2.5}, and get  a non-trivial relation of the form 
$$
c_1 b_1 u_1+\cdots+c_k b_k u_k=0, 
$$
which holds for infinitely many tuples $(b_1,\ldots,b_k, u_1,\ldots,u_k)$  in $\mathfrak{E}$.   
This proves the claim and hence the proposition.
\end{proof}

\section{The key setup and results for the proof of Theorems \ref{maintheorem2}}\label{section}
We define an equivalence relation $\sim$ on $\overline{\mathbb{Q}}^\times$ as follows:
\begin{equation*}\label{eq1}
\tag{1}
\alpha_i\sim \alpha_j~~  \mbox{if there is}~~ \sigma\in G_\mathbb{Q}~~\mbox{such that}~~ \frac{\alpha_i}{\sigma(\alpha_j)}~~~ \mbox{is a root of unity}, 
\end{equation*}
where $G_\mathbb{Q}$ denotes the absolute Galois group over $\mathbb{Q}$.  
\bigskip

We need  the following lemma.
\begin{lemma}\label{lem3}     
Let $(\alpha_1,\ldots,\alpha_k)$  be a non-degenerate tuple of non-zero algebraic numbers. Let $K$ be the Galois closure of $\mathbb{Q}(\alpha_1,\ldots,\alpha_k)$ over $\mathbb{Q}$  and let $r$ be the order of the torsion subgroup of $K^\times$. Then the tuple $(\alpha^r_1,\ldots,\alpha^r_k)$ satisfies the following properties:
\begin{enumerate}
\item[(a)]
For any integer $i$ satisfying $1\leq i\leq k$, if $\beta\neq \alpha^r_i$ is a Galois conjugate to $\alpha^r_i$ over $\mathbb{Q}$, then $\frac{\beta}{\alpha^r_i}$ is not a root of unity.
\item[(b)] For any integers $i$ and $j$ satisfying $1\leq i\neq j\leq k$, if $\frac{\alpha^r_i}{\sigma(\alpha^r_j)}$ is a root of unity for some $\sigma\in\mbox{Gal}(K/\mathbb{Q})$, then $\alpha^r_i=\sigma(\alpha^r_j)$.
\item[(c)] For any integers $i$ and $j$ satisfying $1\le i,j\leq k$, if $\frac{\sigma(\alpha^r_i)}{\rho(\alpha^r_j)}$ is a root of unity for some $\sigma, \rho\in\mbox{Gal}(K/\mathbb{Q})$, then $\sigma(\alpha^r_i)=\rho(\alpha^r_j)$.
\end{enumerate}
\end{lemma}
The proof of this lemma can be easily verified.  

\bigskip


\noindent\textbf{Remark 1.}
  Without loss of generality, we can assume that  the tuple $(\alpha_1,\ldots,\alpha_k)$ satisfies (a), (b) and (c) of Lemma \ref{lem3} with $r=1$; otherwise we work with the tuple  $(\alpha_1^r,\ldots,\alpha_k^r).$ Also, by using part (a) and (b) of Lemma \ref{lem3}, after replacing $(\alpha_1,\ldots,\alpha_k)$ by $(\alpha_1^r,\ldots,\alpha_k^r)$, the equivalence relation defined by (\ref{eq1}) becomes trivial, i.e. $\alpha_i\sim \alpha_j$ if and only if $\alpha_i=\alpha_j.$
\bigskip

Given a set of non-zero algebraic numbers $\alpha_1,\ldots,\alpha_k$, under the equivalence relation given by \eqref{eq1}, we have the following partition
$$
\{\alpha_1,\ldots,\alpha_k\}=\cup_{i=1}^s S_i=\cup_{i=1}^s\{\alpha_{i,1},\ldots,\alpha_{i,m_i}\}.
$$
We also relabel the numbers $\lambda_1,\ldots,\lambda_k$ as $\lambda_{i,1},\ldots,\lambda_{i,m_i}$ for $1\leq i\leq s$. Under these notations, we can express the sum $\sum_{i=1}^k \lambda_i \alpha^n_i$ as  
\begin{equation*}\label{eq2}
\tag{2}
\sum_{i=1}^k \lambda_i \alpha^n_i=\sum_{i=1}^s\sum_{j=1}^{m_i}\lambda_{i,j}\alpha^n_{i,j},
\end{equation*}
and denote the tuple $(n,\lambda_1,\ldots,\lambda_k)$ by $(n, \lambda_{i,j})_{i,j}$. Let $(\alpha_1,\ldots,\alpha_k)$  be a non-degenerate tuple of non-zero algebraic numbers. If needed, by replacing the tuple $(\alpha_1,\ldots,\alpha_k)$ with $(\alpha_1^r,\ldots,\alpha_k^r)$ in \eqref{eq2}, by part (b) of Lemma \ref{lem3}, we can assume that for $1\leq i\leq s$, the elements $\alpha_{i,1},\ldots,\alpha_{i,m_i}$  are Galois conjugate over $\mathbb{Q}$ to each other. We let $d_i\geq m_i$  denote the number of all possible Galois conjugates of $\alpha_{i,1}$ over $\mathbb{Q}$.  We now denote by $\alpha_{i,m_i+1},\ldots,\alpha_{i,d_i}$  all the other conjugates of  $\alpha_{i,1}$ that do not appear in $\{\alpha_{i,1},\ldots,\alpha_{i,m_i}\}$.  For every $\sigma\in \mbox{Gal}(L/\mathbb{Q})$  and $1\leq i\leq s$,  we denote $\sigma(\alpha_{i,j})=\alpha_{i,\sigma_i(j)}$ for $1\leq j\leq d_i$, where $L$ denotes the Galois closure of $K=\mathbb{Q}(\alpha_1,\ldots,\alpha_k, \lambda_1,\ldots,\lambda_k)$ over $\mathbb{Q}$ and  $\{\sigma_i(1),\ldots,\sigma_i(d_i)\}$ is a permutation of $\{1,\ldots,d_i\}$. 
\bigskip

We also need the following lemma from \cite[Lemma 3.1]{kul}.
\begin{lemma}\label{lem2}
Let $K$ be a number field of degree $d$  and let $\{\omega_1, \ldots,\omega_d\}$ be a  $\mathbb{Q}$-basis  for $K$. Then there exist constants $C_1$ and $C_2$ depending only on the $\omega_i$'s such that for every $\alpha\in K$,  we can  write $\alpha=\sum_{i=1}^d b_i\omega_i$, where $b_i\in\mathbb{Q}$ satisfying $h(b_i)\leq C_1 h(\alpha)+C_2$  for $1\leq i\leq d$.
\end{lemma}
The following Proposition is  very crucial for the proof of   Theorem \ref{maintheorem2}.
\begin{prop}\label{propnew}
Let $\alpha_1,\ldots,\alpha_k$,$\lambda_1,\ldots,\lambda_k$,  $f$, $\mathcal{A}$  be as in Theorem \ref{maintheorem2}.
Let $\mathcal{A}_0$ be an infinite subset of $\mathcal{A}$.  Let $p$ be the nearest integer to $q\sum_{i=1}^k\lambda_i\alpha^n_i=q\sum_{i=1}^s\sum_{j=1}^{m_i}\lambda_{i,j}\alpha^n_{i,j}$. Then there exists an infinite subset $\mathcal{A}_1$ of $\mathcal{A}_0$ such that for every tuples $(n,q,\lambda_{i,j})_{i,j}\in\mathcal{A}_0$, we can write $p=q\sum_{i=1}^s\sum_{j=1}^{d_i}\eta_{i,j}\alpha^n_{i,j}$  with the  following properties:
\begin{enumerate}
\item[(i)] $\eta_{i,j}\in L$ and $h(\eta_{i,j})=o(n)$ for $1\leq i\leq s$ and $1\leq j\leq d_i$. 
\item[(ii)] For every $\sigma\in\mathrm{Gal}(L/\mathbb{Q})$  and $1\leq i\leq s$, let $\sigma_i$ denotes the induced permutation on $\{1,\ldots,d_i\}$. Then, we have $\sigma(\eta_{i,j})=\eta_{i,\sigma_i(j)}$ for each pair $(i,j)$ with $1\leq i\leq s$ and $1\leq j\leq d_i$.

\item[(iii)]\label{iii} $\lambda_{i,j}=\eta_{i,j}$  for $1\leq i\leq s$ and $1\leq j\leq m_i.$

\item[(iv)]\label{iv}Let $B$ be the set of $\gamma$ such that $\gamma\notin\{\lambda_{i,j}\alpha^n_{i,j}:1\leq i\leq s, 1\leq j\leq m_i\}$  and $\gamma$ be a Galois conjugate over $\mathbb{Q}$ to $\lambda_{i,j}\alpha^n_{i,j}$  for some pair $(i,j)$  with $1\leq i\leq s, 1\leq j\leq m_i.$  Then the elements $\eta_{i,j}\alpha^n_{i,j}$ for $1\leq i\leq  s$ and $m_i<j\leq d_i$ are distinct and are exactly all the elements of $B$. 
\end{enumerate}
\end{prop}
\begin{proof}

For the proof of this proposition, we proceed along similar lines to the proof of Proposition 3.4 in \cite{kul}. We first observe that along an infinite subset of $\mathcal{A}$, $n$ cannot be fixed. If this is the case, then the assumption that $h(\lambda_i)<f(n)$ implies $h(\lambda_i)$ is bounded for each $1\leq i\leq k$. Since all $\lambda_i$ belongs to a fixed number field $K$, by the Northcott property, there are only finitely many such $\lambda_i.$ Using
$$
0<\Vert\lambda_1 q\alpha^n_1+\cdots+\lambda_k q\alpha^n_k\Vert<\frac{\theta^n}{q^{d+\varepsilon}}<\frac{1}{q^{1+\varepsilon}}
$$ and each $\alpha_i$ is an  algebraic number, by applying Roth's theorem, we conclude that there are only finitely many such tuples $(n, q, \lambda_1, \ldots, \lambda_k)$ with bounded $n$. Hence, $n$ cannot be fixed along infinitely many tuples $(n, q, \lambda_1, \ldots, \lambda_k) \in \mathcal{A}$.

\vspace{.2cm}

For every $v\in M^\infty_L$, fix $\sigma_v\in \mbox{Gal}(L/\mathbb{Q})$  such that $v$ corresponds to the automorphism $\sigma^{-1}_v$. In other words, for every $\alpha\in L$, we have 
\begin{equation*}\label{eq3.1}
\tag{3.1}
|\alpha|_v=|\sigma^{-1}_v(\alpha)|^{d(L)/[L:\mathbb{Q}]},
\end{equation*}
where  $| \cdot |$ denotes the usual complex  absolute value in $\mathbb{C}$ and $d(L)=1$ if $L\subset\mathbb{R}$ and $d(L)=2$ otherwise.  
\smallskip

For $(n, q, \lambda_{i,j})_{i,j}\in \mathcal{A}$,  by Lemma \ref{lem2}, we  write 
\begin{equation*}\label{eq3.2}
\tag{3.2}
\lambda_{i,j}=\sum_{\ell=1}^{d'} b_{i,j,\ell}\omega_\ell
\end{equation*}
where  $b_{i,j,\ell}\in\mathbb{Q}$ and $d'=[\mathbb{Q}(\lambda_1,\ldots,\lambda_k):\mathbb{Q}]$.  Let $p$ be the nearest integer to $q\sum_{i=1}^s\sum_{j=1}^{m_i}\lambda_{i,j}\alpha^n_{i,j}$. Then from \eqref{eq3.2} and \eqref{eq1.2}, we get 
\begin{equation*}\label{eq3.3}
\tag{3.3}
\left|\sum_{i=1}^s\sum_{j=1}^{m_i} q \lambda_{i,j}\alpha^n_{i,j}-p\right|=\left|\sum_{i=1}^s\sum_{j=1}^{m_i}\sum_{\ell=1}^{d'} q\omega_\ell  b_{i,j,\ell}\alpha^n_{i,j}-p\right|<\frac{\theta^n}{q^{d+\varepsilon}}
\end{equation*}
holds for all $(n, q, \lambda_{i,j})_{i,j}\in\mathcal{A}$.
From \eqref{eq3.1} and the formula  $\sum_{v\in M^\infty_L}d(L)=[L:\mathbb{Q}]$, we have 
\begin{align*}
\prod_{v\in M^\infty_L}\left|\sum_{i=1}^s\sum_{j=1}^{m_i}\sum_{\ell=1}^{d'} q \sigma_v(\omega_\ell)  b_{i,j,\ell} \alpha^n_{i,\sigma_{v,i}(j)}-p\right|_v &=\prod_{v\in M^\infty_L}\left|\sum_{i=1}^s\sum_{j=1}^{m_i}\sum_{\ell=1}^{d'} q\omega_\ell  b_{i,j,\ell}\alpha^n_{i,j}-p\right|^{\frac{d(v)}{[L:\mathbb{Q}]}}\\
&=\left|\sum_{i=1}^s\sum_{j=1}^{m_i}\sum_{\ell=1}^{d'} q \omega_\ell  b_{i,j,\ell}\alpha^n_{i,j}-p\right|
\end{align*}
holds for all $(n,q, \lambda_{i,j})_{i,j}\in\mathcal{A}$, where for  each  $v\in M^\infty_L$ and  $1\leq i\leq s$, we have set $\sigma_v(\alpha_{i,j})=\alpha_{i,\sigma_{v,i}(j)}$ and   $\{\sigma_{v,i}(1),\ldots,\sigma_{v,i}(m_i)\}$  is  a permutation of $\{1,\ldots,d_i\}$, and  $d_i$ denotes the degree of  an algebraic number $\alpha_i$. Thus from \eqref{eq3.3}, we have 
\begin{equation*}\label{eq3.4}
\tag{3.4}
\prod_{v\in M^\infty_L}\left|\sum_{i=1}^s\sum_{j=1}^{m_i}\sum_{\ell=1}^{d'} q \sigma_v(\omega_\ell)  b_{i,j,\ell} \alpha^n_{i,\sigma_{v,i}(j)}-p\right|_v<\frac{\theta^n}{q^{d+\varepsilon}}\cdotp
\end{equation*}
Let $\mathfrak{L}:=\{(i,j_1,j_2,\ell):1\leq i\leq s, 1\leq j_1\leq m_i, 1\leq j_2\leq d_i, 1\leq \ell\leq d'\}$.   For each $(n, q, \lambda_{i,j})_{i,j}\in \mathcal{A}$, we associate a vector ${\bf y}:={\bf y}(n,\lambda_{i,j})_{i,j}=(b_{i,j_1,\ell}\alpha^n_{i,j_2}: 1\leq i\leq s, 1\leq j_1\leq m_i, 1\leq j_2\leq d_i, 1\leq \ell\leq d')$, whose  components are indexed by $\mathfrak{L}$  and defined to be $y_{(i,j_1, j_2,\ell)}= b_{i,j_1,\ell}\alpha^n_{i,j_2}$  for $(i,j_1,j_2,\ell)\in\mathfrak{L}$. 
\smallskip

For $v\in M^\infty_L$  and ${\bf a}=(i,j_1,j_2,\ell)\in\mathfrak{L}$, define 
$$
\delta_{v,\bf a}:=\sigma_v(\omega_\ell) \mbox{~~if~} \sigma_{v,i}(j_1)=j_2, ~~\mbox{otherwise~~} 0.
$$
With this notation, the inequality \eqref{eq3.4} can be rewritten as 
\begin{equation*}\label{eq3.5}
\tag{3.5}
\prod_{v\in M^\infty_L}\left|\sum_{\bf a\in \mathfrak{L}}q \delta_{v,\bf a} y_{\bf a}-p\right|_v<\frac{\theta^n}{q^{d+\varepsilon}} \cdotp
\end{equation*}
We choose an infinite subset $\mathcal{A'}$ of $\mathcal{A}$ such that the vector space over $L$ generated by the set of vectors $\{{\bf y}(n,\lambda_{i,j})_{i,j}: (n,q,\lambda_{i,j})_{i,j}\in\mathcal{A'}\}$ has minimal dimension. We denote this vector space by $V$ and let $\mbox{dim}_L(V)=\tau$. Since $b_{i, j_1,\ell}\in\mathbb{Q}$,
by  applying Gaussian elimination to this system of vectors, we obtain a new system of vectors such that number of non-zero entries in each vectors is less than $d_1+\cdots+d_k$. In other words,  we get   a subset $\mathfrak{L}^*$ of $\mathfrak{L}$  consisting of $\tau$ elements with  $\tau\leq d=d_1+d_2+\cdots+d_k$. Then  every vector $Y_{\bf a}\in V$ can be  written as
$$
Y_{\bf a}=\sum_{\bf b\in\mathfrak{L}^*}c_{\bf a, b}Y_{\bf b}\quad \mbox{for all}~~{\bf a}\in \mathfrak{L},
$$
where $c_{\bf a, b}\in L$. Thus for every  ${\bf a}\in \mathfrak{L}\backslash\mathfrak{L}^*$, the corresponding component $y_{\bf a}$ of the vector $Y_{\bf a}$ given by
$
y_{\bf a}=\sum_{\bf b\in\mathfrak{L}^*}c_{\bf a, b}y_{\bf b}.
$
Consequently,  for every $v\in M^\infty_L$, we can write $\displaystyle\sum_{{\bf a}\in \mathfrak{L}}\delta_{v,\bf a} y_{\bf a}=\sum_{\bf b\in\mathfrak{L}^*}\tilde{c}_{v, \bf b}y_{\bf b}$. Therefore the inequality \eqref{eq3.5} can in turn rewrite as  
\begin{equation*}\label{eq3.6}
\tag{3.6}
\prod_{v\in M^\infty_L}\left|\sum_{{\bf b}\in \mathfrak{L}^*} q \tilde{c}_{v, \bf b} y_{\bf b}-p\right|_v<\frac{\theta^n}{q^{d+\varepsilon}}\cdotp
\end{equation*}
Now for each $v\in S$, we define $\tau+1$ linearly independent linear forms in $\tau+1$ variables as follows: for each $v\in M^\infty_L$, let 
\begin{align*}
L_{v,1}({\bf X})&=\sum_{{\bf b}\in \mathfrak{L}^*}\tilde{c}_{v, b} X_{\bf b}-X_1
\end{align*}
and  $L_{v,\bf b}({\bf X})=X_{\bf b}$ for ${\bf b}\in \mathfrak{L}^*$.  If $v\in S\backslash {M^\infty_L}$, define $L_{v,1}({\bf X})=X_1$ and $L_{v,\bf b}({\bf X})=X_{\bf b}$. Clearly,  we see that for each $v\in S$, the above linear forms are linearly independent.
\smallskip

For each $(n,q, \lambda_{i,j})_{i,j}\in\mathcal{A'}$, we define the vector ${\bf x}$  whose coordinates are denoted as $x_1$  and $x_{\bf b}$ for ${\bf b}\in \mathfrak{L}^*$ as follows:
$x_1=p$ and $x_{\bf b}=qy_{\bf b}$. We are now ready to apply the Subspace Theorem, namely Theorem \ref{schli}.  In order to apply  Theorem \ref{schli}, we need to calculate  the following quantity
$$
\prod_{v\in S}\left(\frac{|L_{v,1}({\bf x)}|_v}{||{\bf x}||_v}\prod_{{\bf b}\in \mathfrak{L}^*}\frac{|L_{v,\bf b}{(\bf x)}|_v}{||{\bf x}||_v}\right)\cdotp
$$

 Using the fact that $L_{v,\bf b}({\bf x})=q y_{\bf b}$, for each ${\bf b}\in\mathfrak{L}^*$ and that $\alpha_{i,j}$ are $S$-unit for every pair $(i,j)$, by the product formula
$$
\prod_{v\in S}\prod_{{\bf b}\in \mathfrak{L}^*}|L_{v,\bf b}{(\bf x)}|_v=\prod_{v\in S}\prod_{{\bf b}\in \mathfrak{L}^*} |q|_v \prod_{v\in S}\prod_{(i,j_1,j_2,\ell)\in \mathfrak{L}^*}|b_{i,j_1,\ell}\alpha^n_{i,j_2}|_v\leq \prod_{{\bf b}\in \mathfrak{L}^*} \left(\prod_{v\in S}|q|_v\right) \prod_{{\bf b}\in \mathfrak{L}^*}H(b_{i,j_1,\ell}). 
$$
Let $B=\max\{H(b_{i,j,\ell}):1\leq i\leq s, 1\leq j\leq d_i, 1\leq \ell\leq d'\}$. Then by Lemma \ref{lem2}, we have $B<e^{f(n)}$. Let $N=d'(d_1+\ldots+d_s)$, which is the number of triples $(i,j,\ell)$.
\vspace{.2cm}

Using the formula $\sum_{v\in M^\infty_L}d(v)=[L:\mathbb{Q}]$, $|\mathfrak{L}^*|=\tau$  and the fact that $\tau\leq d$, we get
$$
 \prod_{v\in S}\prod_{{\bf b}\in \mathfrak{L}^*}|L_{v,\bf b}{(\bf x)}|_v\leq \prod_{{\bf b}\in \mathfrak{L}^*}\left(\prod_{v\in S} |q|_v\right) \prod_{{\bf b}\in \mathfrak{L}^*}H(b_{i,j_1,\ell}) \leq (qB)^{|\mathfrak{L}^*|}=q^d B^\tau.
$$ 
From \eqref{eq3.4} and the integrality of $p$, we have 
\begin{equation*}\label{eq3.7}
\tag{3.7}
\prod_{v\in S}\left(\frac{|L_{v,1}({\bf x)}|_v}{||{\bf x}||_{v}}\prod_{{\bf b}\in \mathfrak{L}^*}\frac{|L_{v,\bf b}{(\bf x)}|_v}{
||{\bf x}||_{v}}\right)\leq \frac{q^d \theta^n}{q^{d+\varepsilon}}\frac{B^\tau}{\left(\prod_{v\in S}||{\bf x}||_v\right)^{|\mathfrak{L}^*|+1}}= \frac{ \theta^n}{q^\varepsilon}\frac{B^\tau}{\left(\prod_{v\in S}||{\bf x}||_v\right)^{\tau+1}}\cdotp
\end{equation*}

We estimate the denominator in \eqref{eq3.7} as
$$
\prod_{v\in S}||{\bf x}||_v=\frac{H({\bf x})}{\prod_{v\notin S}||{\bf x}||_v}\geq \frac{H({\bf x})}{\prod_{i,j,\ell}(H(b_{i,j,\ell}))}\geq \frac{H({\bf x})}{B^N}\cdotp
$$
Thus, from \eqref{eq3.7}, we obtain 
\begin{equation*}\label{eq3.8}
\tag{3.8}
\prod_{v\in S}\frac{|L_{v,1}({\bf x)}|_v}{||{\bf x}||_{v}}\prod_{{\bf b}\in \mathfrak{L}^*}\frac{|L_{v,\bf b}{(\bf x)}|_v}{||{\bf x}||_{v}}\leq \frac{\theta^n B^{\tau+N(\tau+1)}}{q^{\varepsilon}H({\bf x})^{\tau+1}}\cdotp
\end{equation*}
Notice that 
\begin{align*}
H({\bf x})&=\prod_{v\in M_L}\max\{ |p|_v, |q y_{\bf b}|_v:~  {\bf b}\in \mathfrak{L}^*\}\leq \max\{|p|, |q|\}\prod_{v\in M_L}\max\{1,|y_{\bf b}|_v:~  {\bf b}\in \mathfrak{L}^*\}\\
&\leq \max\{|p|, |q|\}\prod_{v\in S}\max\{1,|\alpha^n_{i,j}|_v:~1\leq i\leq s, 1\leq j\leq d_i\}\times\\
&\hspace{.4cm}\prod_{v\in M_L}\max\{1, |b_{i,j,\ell}|:1\leq i\leq s, 1\leq j\leq d_i, 1\leq \ell\leq d'\}\\
&=\max\{|p|, |q|\}\prod_{i,j}H(\alpha_{i,j})^n \cdot \prod_{i,j,\ell} H(b_{i,j,\ell})\leq \max\{|p|, |q|\} B^{N}\prod_{i,j}H(\alpha_{i,j})^n.
\end{align*}
Since $H(\lambda_{i,j})<e^{f(n)}$ for $1\leq i\leq s$ and $1\leq j\leq m_i$, from Lemma \ref{lem2}, we have $H(b_{i,j,\ell})< e^{f(n)}$ for $1\leq \ell\leq d'$.  By \eqref{eq3.3}, we have 
$$
|p|\leq \left|q\sum_{i=1}^s\sum_{j=1}^{m_i}\lambda_{i,j}\alpha^n_{i,j}\right|+1\leq |q| C^n
$$
for sufficiently large $n$ along the tuples $(n, q, \lambda_{i,j})_{i,j}\in\mathcal{A}'$ and some constant $C>1$ depending on number field $K$,  $\alpha_{i,j}$'s and the sublinear function $f$. 
Now, by combining both the above estimates together with the fact that $n\to\infty$ along the infinite set $\mathcal{A}'$, we deduce that 
\begin{equation*}\label{eq3.9}
\tag{3.9}
H({\bf x})=\prod_{v\in M_L}\max\{ |p|_v, |q y_{\bf b}|_v:~  {\bf b}\in \mathfrak{L}^*\}< q C^n_1
\end{equation*}
for all sufficiently large values of $n$ and some constant $C_1>1$ depending on $H(\alpha_{i,j})$'s and $B$.  
\bigskip

Using $\theta\in (0,1)$ and  $B<e^{f(n)}$, where $f(n)$ is a sub-linear function, together with  \eqref{eq3.8},  there exists $0<\theta<\theta'<1$ such that 
$$
\prod_{v\in S}\left(\frac{|L_{v,1}({\bf x)}|_v}{||{\bf x}||_{v}}\prod_{{\bf b}\in \mathfrak{L}^*}\frac{|L_{v,\bf b}{(\bf x)}|_v}{||{\bf x}||_{v}}\right)\leq \frac{\theta^n B^{\tau+N(\tau+1)}}{ q^\varepsilon H({\bf x})^{\tau+1}}\leq \frac{\theta'^n}{ q^\varepsilon H({\bf x})^{\tau+1}}
$$
for all $n$ sufficiently large along $(n,q,\lambda_{i,j})_{i,j}\in\mathcal{A}'$. Choose $\varepsilon'>0$ such that $0<\varepsilon'<-\frac{\log \theta'}{\log C_1}$. Then we have 
$$
\prod_{v\in S}\left(\frac{|L_{v,1}({\bf x)}|_v}{||{\bf x}||_{v}}\prod_{{\bf b}\in \mathfrak{L}^*}\frac{|L_{v,\bf b}{(\bf x)}|_v}{||{\bf x}||_{v}}\right)\leq\frac{\theta'^n}{ q^\varepsilon H({\bf x})^{\tau+1}}\leq \frac{1}{ C_1^{\varepsilon' n}q^\varepsilon H({\bf x})^{\tau+1}}
$$
for all $n$ sufficiently large along $(n,q,\lambda_{i,j})_{i,j}\in\mathcal{A}'$.  Now we set $\varepsilon''=\min \{\varepsilon, \varepsilon'\}$. Thus from \eqref{eq3.9}, we conclude that  
$$
\prod_{v\in S}\left(\frac{|L_{v,1}({\bf x)}|_v}{||{\bf x}||_{v}}\prod_{{\bf b}\in \mathfrak{L}^*}\frac{|L_{v,\bf b}{(\bf x)}|_v}{||{\bf x}||_{v}}\right)\leq\frac{1}{ C_1^{\varepsilon' n}q^\varepsilon H({\bf x})^{\tau+1}}\leq \frac{1}{H({\bf x})^{\tau+1+\varepsilon''}}
$$
for all $n$ sufficiently large along $(n,q,\lambda_{i,j})_{i,j}\in\mathcal{A}'$. 
By Theorem \ref{schli}, there exists  a non-trivial relation of the form 
\begin{equation*}\label{eq3.10}
\tag{3.10}
a_1 p+q \sum_{{\bf b}\in \mathfrak{L}^*}a_{\bf b} y_{\bf b}=0, \quad a_1, a_{\bf b}\in L  
\end{equation*}
holds for infinitely  many  $(n,q, \lambda_{i,j})_{i,j}\in\mathcal{A}'$. First, we observe that $a_1\neq 0$. Indeed, suppose we have $a_1=0$. Then, the non-trivial relation contradicts the minimality of $\tau$.  By the definition of $y_{\bf b}$, \eqref{eq3.10} can be written as
$$
a_1 p+q \sum_{i=1}^s \sum_{j_2=1}^{d_i}\left(\sum_{j_1=1}^{m_i}\sum_{\ell=1}^{d'} a_{i,j_1,j_2,\ell} b_{i, j_1,\ell} \right)\alpha^n_{i,j_2}=0.
$$ 
Using the fact that $H(b_{i,j_1,\ell})<e^{f(n)}$ and $a_1$ is non-zero, we can re-write the above relation as 
\begin{equation*}\label{eq3.11}
\tag{3.11}
p=q\sum_{i=1}^s\sum_{j=1}^{d_i}\eta_{i,j}\alpha^n_{i,j},
\end{equation*}
where $\eta_{i,j}\in L$ and $H(\eta_{i,j})<e^{f(n)}$ for every pair $(i,j)$.  This proves the part (i) of the proposition.
\smallskip

For any $\sigma\in\mathrm{Gal}(L/\mathbb{Q})$, we have 
$$
p=q\sum_{i=1}^s\sum_{j=1}^{d_i}\sigma(\eta_{i,j})\alpha^n_{i,\sigma_{i}(j)}=q\sum_{i=1}^s\sum_{j=1}^{d_i}\sigma(\eta_{i,\sigma^{-1}_i(j)})\alpha^n_{i,j}
$$
where $\sigma_i^{-1}$ is the permutation of $\sigma_i$  on $\{1,\ldots,d_i\}$ induced by $\sigma$. Together with \eqref{eq3.11}, we obtain
$$
\sum_{i=1}^s\sum_{j=1}^{d_i}(\eta_{i,j}-\sigma(\eta_{i,\sigma^{-1}_i(j)}))\alpha^n_{i,j}=0\quad\mbox{for}~~~~\sigma\in \mbox{Gal}(L/\mathbb{Q}).
$$
Now the proof of part $(ii)$ of this proposition, we conclude exactly as part $(ii)$ of Proposition 3.4 in \cite{kul}.
\vspace{.2cm}

Set $\lambda_{i,j}=0$ for $1\leq i\leq s,$  $m_i< j\leq d_i$. Substitute  the value of $p$ from \eqref{eq3.11} into \eqref{eq3.3}, we get 
\begin{equation*}\label{eq3.12}
\tag{3.12}
0<\left|\sum_{i=1}^s\sum_{j=1}^{d_i}(\lambda_{i,j}-\eta_{i,j})\alpha_{i,j}^n\right|<\frac{\theta^n}{q^{d+1+\varepsilon}}<\frac{\theta^n}{q}<\theta^n
\end{equation*}
holds for all  $(n,\lambda_{i,j})_{i,j}$ along the tuples  $(n, q,\lambda_{i,j})_{i,j}\in\mathcal{A}'$.  Now we claim that for any infinite subset $\mathcal{A}''$ of $\mathcal{A}'$, we have $\lambda_{i,j}-\eta_{i,j}= 0$ in \eqref{eq3.12} for $1\leq i\leq s$, $1\leq j\leq m_{i}$.  Suppose there is an infinite subset $\mathcal{B}$  of $\mathcal{A}'$ such that $\lambda_{i,j}-\eta_{i,j}\neq 0$  for $1\leq i\leq s$, $1\leq j\leq m_{i}$. Then, there is a  set $\mathcal{P}$ given by 
$$
\mathcal{P}:=\{(i,j):1\leq i\leq s, 1\leq j\leq d_i, \lambda_{i,j}-\eta_{i,j}\neq 0\}.
$$
By  Lemma \ref{lem3}, WLOG we can assume that the tuple $(\alpha_{i,j};(i,j)\in\mathcal{P})$ is non-degenerate, otherwise since the sum $\sum_{i=1}^s\sum_{j=1}^{d_i}(\lambda_{i,j}-\eta_{i,j})\alpha_{i,j}^n$ is non-zero, by Lemma \ref{lem3} we can reduce the tuple $(\alpha_{i,j}:(i,j)\in\mathcal{P})$  to a maximal length of non-degenerate tuple. Since $|q\alpha^n_{i,j}|\geq 1$ and  the fact that $H(\lambda_{i,j}-\eta_{i,j})=H(1/(\lambda_{i,j}-\eta_{i,j}))=e^{o(n)}$,   for $(i,j)\in\mathcal{P}$ and any $\delta\in (0,1)$, we have that 
\begin{equation*}
\max\{|(\lambda_{i,j}-\eta_{i,j})q\alpha^n_{i,j}|:(i,j)\in\mathcal{P}\}\geq \max\{|(\lambda_{i,j}-\eta_{i,j})|:(i,j)\in\mathcal{P}\}>\delta^n
\end{equation*}
for all  $n$  sufficiently large.   Choose $\varepsilon>0$ such that 
\begin{equation*}\label{eq3.13}
\tag{3.13}
\theta^n<\frac{\max\{|(\lambda_{i,j}-\eta_{i,j})\alpha^n_{i,j}|: (i,j)\in\mathcal{P}\}}{\left(\prod_{i,j\in\mathcal{P}} H(\lambda_{i,j}-\eta_{i,j})\right)^{
 |\mathcal{P}|+1+\varepsilon}H(\alpha^n_{i,j}:(i,j)\in\mathcal{P})^\varepsilon}.
\end{equation*}

By Proposition 2.3 from \cite{kul} to \eqref{eq3.12} together with \eqref{eq3.13}, we get a non-trivial relation among $(\lambda_{i,j}-\eta_{i,j})\alpha^n_{i,j}$ for infinitely many $(n,\lambda_{i,j})_{i,j}$ along the tuples  $(n, q,\lambda_{i,j})\in\mathcal{A}$.  Proposition \ref{prop-kul1} now leads  to a contradiction. Hence, we conclude that $\lambda_{i,j}=\eta_{i,j}$ for $1\leq i\leq s$, $1\leq j\leq m_i$ and along the infinite set $\mathcal{B}$. This proves part $(iii)$.  The proof of part $(iv)$ of this proposition follows exactly as \cite[Proposition 3.4, part (iv)]{kul}, so we omit the proof here.
\end{proof}

\section{\bf Proof of Theorem \ref{maintheorem2}}\label{sec4}
Let $\mathcal{A}$ be the infinite set of tuples $(n,q,\lambda_1,\ldots,\lambda_k)\in\mathbb{N}^2\times (K^\times)^k$ satisfying \eqref{eq1.2}, where $K$ is a number field. By extending $K$, we may assume that $K=\mathbb{Q}(\lambda_1,\ldots,\lambda_k,\alpha_1,\ldots,\alpha_k)$. Let $L$ be its Galois closure over $\mathbb{Q}$, $r$ be the order of the torsion subgroup of $L^\times$, and $G=\mbox{Gal}(L/\mathbb{Q})$ be the Galois group of $L$ over $\mathbb{Q}$. By replacing, if needed, $(\alpha_1,\ldots,\alpha_k)$ with $(\alpha^r_1,\ldots,\alpha^r_k)$ and taking $n\equiv a(\mbox{mod}~ r)$ for some $0\leq a\leq r-1$, without loss of generality, we can assume that the tuple $(\alpha_1,\ldots,\alpha_k)$ satisfies (a), (b), and (c) of Lemma \ref{lem3} with $r=1.$ Let $S$ be a suitable finite subset of $M_L$ containing all the archimedean places such that $\alpha_i$ is an $S$-unit for each $i=1,2,\ldots,k$, and stable under Galois conjugation. The notations $s, m_i, d_i, \alpha_{i,j}$ and $\lambda_{i,j}$ are as  introduced in Section \ref{section}.

\subsection{Proof of Property (i) of Theorem \ref{maintheorem2}} We want to prove that at least one of $\alpha_i$ is an algebraic integer. Let $\mathcal{A}_1$ be an infinite subset of $\mathcal{A}$ satisfying the conclusion of Proposition \ref{propnew}. Let $p$ be the nearest integer to $q\sum_{i=1}^k \lambda_i \alpha^n_i$. Then for every $(n,q, \lambda_{i,j})_{i,j}\in \mathcal{A}_1$,  we can write
\begin{equation*}\label{eq4.1}
\tag{4.1}
p=\sum_{i=1}^s\sum_{j=1}^{d_i}q\eta_{i,j}\alpha^n_{i,j}
\end{equation*}
with $H(\eta_{i,j})<e^{f(n)}$. 
First, we note that $[\mathbb{Q}(\alpha_1,\ldots,\alpha_k):\mathbb{Q}]=d\geq 2$. Suppose that $\alpha_i=\frac{a_i}{b_i}\in\mathbb{Q}$ for $1\leq i\leq k$. Then $d_i=1$ for $1\leq i\leq s$, and hence \eqref{eq4.1},  we can write 
$$
p=q\sum_{i=1}^s\eta_{i,1} \left(\frac{a_{i,1}}{b_{i,1}}\right)^n
$$
with $H(\eta_{i,1})=e^{o(n)}$.
Using part (iii) of Proposition \ref{propnew}, further $p$ can be written as 
$$
p=q\sum_{i=1}^s\lambda_{i,1} \left(\frac{a_{i,1}}{b_{i,1}}\right)^n. 
$$
Since $p$ is the nearest integer to $q\sum_{i=1}^k \lambda_i \alpha^n_i$, substituting this into \eqref{eq1.2} leads to a contradiction. Therefore, we can assume that $d\geq 2$. Consequently, substituting the value of $p$ from \eqref{eq4.1} into \eqref{eq1.2} and using part (iii) of Proposition \ref{propnew}, we obtain that
\begin{equation*}\label{eq4.2}
\tag{4.2}
0<\left|\sum_{i=1}^s\sum_{j=m_i+1}^{d_i}\eta_{i,j}\alpha^n_{i,j}\right|<\frac{\theta^n}{q^{d+1+\varepsilon}}\leq \theta^n.
\end{equation*}
Assume that none of $\alpha_i$ is an algebraic integer. Since $d\geq 2$, so there exists at least one $i_0\in\{1,\ldots,k\}$ such that the Galois conjugate of $\alpha_{i_0}$ other than itself  does appear in \eqref{eq4.2}.  Then there is a  finite  place $\omega$ and $j_0$ with  $m_{i_0+1}\leq j_0\leq d_{i_0}$ such that $|\alpha_{i_0, j_0}|_\omega>1$, which in turn entails that
$$
\max\{|\eta_{i,j}\alpha^n_{i,j}|_\omega:1\leq i\leq s, m_i+1\leq j\leq d_i\}>1,
$$
for all  $n$ is sufficiently large along the tuples $(n,q,\lambda_{i,j})\in\mathcal{A}_1$. Using this lower bound, we can choose $\varepsilon>0$ such that 
$$
\theta^n<\frac{\max\{|\eta_{i,j}\alpha^n_{i,j}|_\omega:1\leq i\leq s, m_i+1\leq j\leq d_i\}}{\left(\prod_{i,j} H(\eta_{i,j})\right)^{
 N+1+\varepsilon}H(\alpha^n_{i,j} : 1\leq i\leq s, m_i<j\leq d_i)^\varepsilon}, 
$$
where $N=\sum_{i=1}^s(d_i-m_i)$. By  Lemma \ref{lem3}, first we  reduce the sum $\sum_{i=1}^s\sum_{j=m_i+1}^{d_i}\eta_{i,j}\alpha^n_{i,j}$ to a  non-degenerate sum,  and then apply  Proposition 2.3 in \cite{kul} with this above choice of  $\varepsilon$ and Proposition \ref{prop-kul1} exactly as we have seen earlier to arrive at a contradiction. Thus, we conclude that at least one of $\alpha_i$ is an algebraic integer.
\bigskip

Now, our aim is to show that each $\alpha_i$ is an algebraic integer under the hypothesis that $h(q\lambda_i)<f(n)$ for $i=1,\ldots,k$. Suppose that  $\alpha_i$ is not an algebraic integer for some  $i\in\{1,\ldots,k\}$. Without loss of generality, we can assume that $\alpha_{1,1}$  is not an algebraic integer (after relabeling).  Since $\alpha_i$ is $S$-unit for each $i = 1,2,\ldots, k$,    there exists a finite place $\omega\in S$  such that $|\alpha_{1,1}|_\omega>1$.    We proceed to get a contradiction. 
\smallskip

By parts (iii) and (iv) of Proposition \ref{propnew},  $\eta_{i,j}\neq 0$ for every pair $(i,j)$ in \eqref{eq4.1}.  Using the fact that $p$ is a non-zero integer (the proof of this fact is given  in the proof of Property $(iv)$ of this theorem),   we have 
\begin{equation*}\label{eq4.3}
\tag{4.3}
0<\left|\sum_{i=1}^s\sum_{j=1}^{d_i}q\eta_{i,j}\alpha^n_{i,j}\right|_\omega\leq 1.
\end{equation*}
Using the fact that  $H(\alpha)=H(\alpha^{-1})$ for every non-zero algebraic number $\alpha$ and the hypothesis $H(q\eta_{i,j})=e^{o(n)}$, we have 
$$
H(q \eta_{1,1})=H(q^{-1}\eta^{-1}_{1,1})=e^{o(n)},
$$
which in turns implies that for every $\delta\in (0,1)$, $|q\eta_{1,1}|_\omega>\delta^n$ for all $n$ sufficiently large, and hence  
\begin{equation*}\label{eq4.4}
\tag{4.4}
\max\{|q\eta_{i,j}\alpha^n_{i,j}|_\omega:1\leq i\leq s, 1\leq j\leq d_i\}>\delta^n |\alpha_{1,1}|^n_\omega>1.
\end{equation*}
From inequalities \eqref{eq4.3} and \eqref{eq4.4}, we can apply Proposition 2.3 from  \cite{kul} with an appropriate choice of $\varepsilon$ (such choice can be made with the help \eqref{eq4.4}) to get a non-trivial relation in $\eta_{i,j}\alpha^n_{i,j}$'s for infinitely many tuples $(n, \lambda_{i,j})_{i,j}$ along the tuples $(n,q,\lambda_{i,j})_{i,j}\in \mathcal{A}$. Then by Lemma \ref{lem3} and Proposition \ref{prop-kul1}, we get a contradiction.

\subsection{Proof of Property (ii) of Theorem \ref{maintheorem2}}   Let $\sigma\in \mbox{Gal}(L/\mathbb{Q})$ and  $i \in \{1,\ldots,k\}$ be such that $\frac{\sigma(\alpha_i)}{\alpha_j}$ is not a root of unity for $j=1,\ldots,k$. Then we prove that $|\sigma(\alpha_i)|<1$, which is equivalent to show that $|\alpha_{i,j}|<1$ for $m_i<j\leq d_i$.  Assume that  $|\alpha_{i,j_0}|\geq 1$ for some $m_i<j_0\leq d_i$. Let $\mathcal{A}_1$ be an infinite subset of $\mathcal{A}$ satisfying the conclusion of Proposition \ref{propnew}. Then using the fact that $H(\eta_{i,j})<e^{f(n)}=e^{o(n)}$, we have 
\begin{equation*}\label{eq4.5}
\tag{4.5}
\max\{|\eta_{i,j}\alpha^n_{i,j}|:1\leq i\leq s, m_i<j\leq d_i\}\geq |\eta_{i,j}|>\delta^n
\end{equation*}
for every $\delta\in (0,1)$ and for all $n$ sufficiently large along the tuples $(n,\lambda_{i,j})_{i,j}\in\mathcal{A}_1$.
\smallskip

 For each tuple $(n,q,\lambda_{i,j})_{i,j}\in\mathcal{A}_1$, we can write $p=q\sum_{i=1}^s\sum_{j=1}^{d_1}\eta_{i,j}\alpha^n_{i,j}$, and substituting it  into \eqref{eq1.2}, we have
\begin{equation*}\label{eq4.6}
\tag{4.6}
0<\left|\sum_{i=1}^s\sum_{j=m_i+1}^{d_i}\eta_{i,j}\alpha^n_{i,j}\right|<\frac{\theta^n}{q^{d+1+\varepsilon}}\leq \theta^n
\end{equation*}
for all but finitely many $(n, \lambda_{i,j})_{i,j}$ along the tuple $(n, q, \lambda_{i,j})_{i,j}\in\mathcal{A}_1$. By  Lemma \ref{lem3}, we can reduce the sum $\sum_{i=1}^s\sum_{j=m_i+1}^{d_i}\eta_{i,j}\alpha^n_{i,j}$ to the non-degenerate sum,  and  then apply Proposition \ref{prop-kul2} to the inequality \eqref{eq4.6} with an appropriate choice of $\varepsilon$ with the help \eqref{eq4.5}  exactly  as in  part (iii) of Proposition \ref{propnew}   to get a required contradiction. 
\subsection{Proof of Property (iii) of Theorem \ref{maintheorem2}} 
 Let $\sigma\in \mbox{Gal}(L/\mathbb{Q})$ and a pair $(i,j)\in \{1,\ldots,k\}^2$ such that $\sigma(\lambda_i\alpha^n_i)=\lambda_j \alpha^n_j$ holds for all but finitely many $(n,\lambda_1,\ldots,\lambda_k)$ along the tuples $(n,q, \lambda_1,\ldots,\lambda_k)\in\mathcal{A}$. Using the fact that $H(\lambda_j)<e^{f(n)}$ for all $j$ and the relation  $\frac{\sigma(\alpha_i)^n}{\alpha^n_j}=\frac{\lambda_j}{\sigma(\lambda_i)}$ for infinitely many tuples $(n,\lambda_1,\ldots,\lambda_k)$, we conclude that $\frac{\sigma(\alpha_i)}{\alpha_j}$ is  a root of unity.  For the converse part, it suffices to prove the following claim:  if $\frac{\sigma(\alpha_{i_1, j_1})}{\alpha_{i_2, j_2}}$ is a root of unity for some $\sigma\in\mbox{Gal}(L/\mathbb{Q})$ and  pairs $(i_1,j_1)$, $(i_2, j_2)$ with $1\leq i_1, i_2\leq s$ and $1\leq j_1\leq m_i$, $1\leq j_2\leq m_{i_2}$, then $\sigma(\lambda_{i_1, j_1}\alpha^n_{i_1, j_1})=\lambda_{i_2, j_2} \alpha^n_{i_2,j_2}$ holds for all but finitely many  $(n,\lambda_{i_1, j_1}, \lambda_{i_2,j_2})$ along the tuples $(n,q,\lambda_{i,j})_{i,j} \in\mathcal{A}$. Suppose there exists an infinite subset $\mathcal{A}_1$ of $\mathcal{A}$ such that $\sigma(\lambda_{i_1, j_1}\alpha^n_{i_1, j_1})\neq\lambda_{i_2, j_2} \alpha^n_{i_2,j_2}$ for all $(n,\lambda_{i_1,j_1},\lambda_{i_2,j_2})$ along the set $\mathcal{A}_1.$ Since $\frac{\sigma(\alpha_{i_1, j_1})}{\alpha_{i_2,j_2}}$ is a root of unity,  by Lemma \ref{lem3},  we have  $i_1=i_2$ and  $\sigma_i(j_1)=j_2$. Let $\mathcal{A}_0$ is a subset of $\mathcal{A}$ satisfying Proposition \ref{propnew}. Then we can write $p=\sum_{i=1}^s\sum_{j=1}^{d_i}q\eta_{i,j}\alpha^n_{i,j}$. By part (iii) of Proposition \ref{propnew}, we have
$$
\sigma(\lambda_{i_1,j_1}\alpha^n_{i_1,j_1})=\sigma(\eta_{i_1,j_1}\alpha^n_{i_1,j_1})=\eta_{i_1,j_2}\alpha^n_{i_1, j_2}=\lambda_{i_1,j_2}\alpha^n_{i_1,j_2},
$$
which contradicts the choice of $\mathcal{A}_0$ and hence the assertion.
\subsection{Proof of Property (iv) of Theorem \ref{maintheorem2}} Assume there is an infinite subset $\mathcal{A}_0$ of $\mathcal{A}$ such that $(\lambda_1 q\alpha^n_1,\ldots,\lambda_k q\alpha^n_k)$ is not pseudo-Pisot for every $(n, q, \lambda_1,\ldots,\lambda_k)\in\mathcal{A}_0$. Let $\mathcal{A}_1$ be an infinite subset of $\mathcal{A}_0$ satisfying the conclusion of Proposition \ref{propnew}. Then for every $(n,q,\lambda_{i,j})_{i,j}\in \mathcal{A}_1$,  we can write 
$$
p=q\sum_{i=1}^s\sum_{j=1}^{d_i}\eta_{i,j}\alpha^n_{i,j}. 
$$
Substituting this value of $p$ into \eqref{eq1.2} and using part (iii) of Proposition \ref{propnew}, we get
\begin{equation*}\label{eq4.7}
\tag{4.7}
0<\left|\sum_{i=1}^s\sum_{j=m_i+1}^{d_i}\eta_{i,j}\alpha^n_{i,j}\right|<\frac{\theta^n}{q^{d+1+\varepsilon}}.
\end{equation*} 
Now we claim that 
$$
\max\{|q\eta_{i,j}\alpha^n_{i,j}|:1\leq i\leq s, m_i+1\leq j\leq d_i\}<1
$$
for all but finitely many  tuple $(n, q, \lambda_{i,j})_{i,j}\in\mathcal{A}_1$. Suppose we have 
\begin{equation*}\label{eq4.8}
\tag{4.8}
\max\{|\eta_{i,j}\alpha^n_{i,j}|:1\leq i\leq s, m_i+1\leq j\leq d_i\}\geq \frac{1}{q} 
\end{equation*}
for infinitely many $(n,q, \lambda_{i,j})_{i,j}\in\mathcal{A}_1$.   
Choose  $\varepsilon>0$ such that 
$$
\frac{\theta^n}{q}<\frac{\max\{|\eta_{i,j}\alpha^n_{i,j}|_\omega:1\leq i\leq s, m_i+1\leq j\leq d_i\}}{\left(\prod_{i,j} H(\eta_{i,j})\right)^{
 N+1+\varepsilon}H(\alpha^n_{i,j} : 1\leq i\leq s, m_i<j\leq d_i)^\varepsilon}, 
$$
where $N=\sum_{i=1}^s(d_i-m_i)$.
By repeating the same  argument as we have seen before, we arrive at a contradiction and hence prove the claim. Thus, using the above claim and part $(iv)$ of Proposition \ref{propnew}, we conclude that the tuple $(\lambda_{i,j}q\alpha^n_{i,j}:1\leq i\leq s, 1\leq j\leq m_i)$ is pseudo-Pisot for all but finitely many $(n,q, \lambda_{i,j})_{i,j}\in\mathcal{A}_1$. This contradicts the choice of $\mathcal{A}_0$. Therefore, the tuple $(\lambda_1 q\alpha^n_1,\ldots,\lambda_k q\alpha^n_k)$ is pseudo-Pisot for all but finitely many tuples $(n,q,\lambda_1,\ldots,\lambda_k)\in\mathcal{A}$.
\bigskip

For the non-vanishing part of this property, we proceed first by  noticing that $p$ is non-zero. Indeed, if this is not the case, from \eqref{eq1.2}, we have  
$$
0<|\lambda_1\alpha^n_1+\cdots+\lambda_k\alpha^n_k|<\frac{\theta^n}{q^{d+1+\varepsilon}}
$$
holds for all $(n,q, \lambda_{i,j})_{i,j}\in\mathcal{A}_1$. Using the facts that  $H(\lambda_{i,j})<e^{f(n)}$ and the tuple $(\alpha_1,\ldots,\alpha_k)$ is non-degenerate,  by Proposition 2.3 in \cite{kul} and Proposition \ref{prop-kul1},  we arrive at a contradiction. Hence, $p$ is non-zero.

Now we claim that no proper subsum of the sum   $\sum_{i=1}^s\sum_{j=1}^{d_i}\eta_{i,j}\alpha^n_{i,j}$ is zero.  Let $\mathcal{P}$ be a set of pairs $(i,j)$ such that the sum 
\begin{equation*}\label{eq4.9}
\tag{4.9}
\sum_{(i,j)\in\mathcal{P}}\eta_{i,j}\alpha^n_{i,j}=0
\end{equation*}
for all but finitely many $(n, \lambda_{i,j})_{i,j}$ along the tuple $(n, q,\lambda_{i,j})_{i,j}\in\mathcal{A}_1$. Then, we derive a contradiction. By  Proposition \ref{prop-kul1},   there exist pairs $(i_1, j_1)\neq (i_2, j_2)$ such that the quotient $\frac{\alpha_{i_1,j_1}}{\alpha_{i_2, j_2}}$ is a root of unity.  We claim  that $\alpha_{i_1, j_1}, \alpha_{i_2, j_2}\in\{\alpha_{i,j}:1\leq i\leq s, 1\leq j\leq m_i\}$. If this is not the case, then there exist automorphisms $\sigma,\rho\in\mbox{Gal}(L/\mathbb{Q})$ such that $\frac{\sigma(\alpha_{i^*_1, j^*_1})}{\rho(\alpha_{i^*_2, j^*_2})}=\frac{\alpha_{i_1, j_1}}{\alpha_{i_2,j_2}}$ for some $\alpha_{i^*_1, j^*_1}, \alpha_{i^*_2, j^*_2}\in\{\alpha_{i,j}:1\leq i\leq s, 1\leq j\leq m_i\}$. Since the quotient $\frac{\alpha_{i_1,j_1}}{\alpha_{i_2,j_2}}$ is a root of unity, by property (iv) of this theorem, we have $\rho^{-1}\circ \sigma(\lambda_{i^*_1,j^*_1}\alpha^n_{i^*_1,j^*_1})=\lambda_{i^*_2, j^*_2}\alpha^n_{i^*_2,j^*_2}$ for all but finitely many $n$ along the tuples $(n,q,\lambda_{i,j})_{i,j}\in\mathcal{A}_1$. This implies that 
$$
\lambda_{i_1, j_1}\alpha^n_{i_1,j_1}=\lambda_{i_2, j_2}\alpha^n_{i_2,j_2}
$$
for all but finitely many $n$ along the tuples $(n,q,\lambda_{i,j})_{i,j}\in\mathcal{A}_1$.
 
Then from part (iii) of Proposition \ref{propnew}, we get that $\alpha_{i_1, j_1}, \alpha_{i_2, j_2}\in\{\alpha_{i,j}:1\leq i\leq s, 1\leq j\leq m_i\}$,  and due to the fact that the tuple $(\alpha_{i,j} : 1\leq i\leq s, 1\leq j\leq m_i)$ is non-degenerate, we conclude that \eqref{eq4.9} does not hold for infinitely many tuples $(n, \lambda_{i,j})_{i,j}$. Thus, no proper subsum of the sum $\sum_{i=1}^s\sum_{j=1}^{d_i}q\eta_{i,j}\alpha^n_{i,j}$ is zero for infinitely many tuples $(n, q \lambda_{i,j})_{i,j}\in\mathcal{A}$.  

\section{\bf Proof of Corollary \ref{cor}}
Let $L = \mathbb{Q}(\lambda, \alpha)$ be the number field, and $K$ be its Galois closure over $\mathbb{Q}$. Let $r$ be the order of the torsion subgroup of $K^\times$. Let $\mathcal{A}$ be an infinite set of triples $(n,q, \lambda)\in \mathbb{N}^2\times K^\times$ satisfying   
$$
0<\Vert\lambda q \alpha^n\Vert<\frac{\theta^n}{q^{d+\varepsilon}}\quad\mbox{and~~} h(\lambda)<f(n).
$$
Since $\mathcal{A}$ is infinite, there exists an integer $a \in \{0, 1, \ldots, r-1\}$ such that $n = a+rm$ for infinitely many natural numbers $m$. Let the collection of such   triples $(n, q, \lambda)$  be $\mathcal{A}'$.
 By properties (i)  and (iii) of Theorem \ref{maintheorem2} with the inputs $\lambda_1=\lambda \alpha^a$,  $k=1$ and $\mathcal{A}$,  we get that $\alpha^r$ is an algebraic integer and $\lambda q\alpha^n$ is the pseudo-Pisot number for all but finitely many $(n,q, \lambda)\in\mathcal{A}'$  such that $n=a+rm$.  In order to complete the proof of this corollary,   it suffices to show that $|\sigma(\alpha^r)|<1$ for each embedding $\sigma\neq \text{Id}:\mathbb{Q}(\alpha^r)\to\mathbb{C}$. We first observe that any conjugate $\sigma(\alpha^r)\neq \alpha^r$ has an absolute value less than or equal to $1$.   Assume that  $|\sigma(\alpha^r)|>1$.  Since   $\lambda q\alpha^n$ is pseudo-Pisot number, we must have $\rho(\lambda_1 q\alpha^{rm})=\rho(\lambda_1)q\sigma(\alpha^r)^m=\lambda_1 q\alpha^{rm}$ for all but finitely many triples $(m,q,\lambda)$ along the triples $(n,q, \lambda)\in\mathcal{A}$ and some  $\rho\in\mbox{Gal}(L/\mathbb{Q})$, where $\sigma$ is the  restriction of the automorphism $\rho$ on $\mathbb{Q}(\alpha^r)$. Then, by property (iv) of Theorem \ref{maintheorem2}, we have that $\sigma(\alpha^r)/\alpha^r$ is a root of unity. Since $r$ is the order of the torsion subgroup, we must have $\sigma(\alpha^r) = \alpha^r$, which is a contradiction. Thus, we conclude that $|\sigma(\alpha^r)|\leq 1$. Now we show that the possibility $|\sigma(\alpha^r)|=1$ cannot occur.  If we have $|\sigma(\alpha^r)|=1$, then the quotient $\frac{\sigma(\alpha^r)}{\alpha^r}$ is not a root of unity. By property (ii) of Theorem \ref{maintheorem2}, we have $|\sigma(\alpha^r)|<1$, which is a contradiction to the assumption that $|\sigma(\alpha^r)|>1$.  This proves that $|\sigma(\alpha^r)|<1$ for each embedding $\sigma\neq \text{Id}:\mathbb{Q}(\alpha^r)\to\mathbb{C}$, and hence finishes the proof of Corollary \ref{cor}.
\section{Applications}\label{sec5}

In this section, we prove two transcendence results for certain infinite products of algebraic numbers as an application of Corollary \ref{cor}. The first result is the following, which generalizes earlier result of \cite[Theorem 1]{cor}.
\begin{theorem}\label{th6.2}
    Let $\alpha>1$ be a real algebraic number of degree $d$ such that no power of $\alpha$ is a Pisot number. Let $(a_n)_n$ be a sequence of positive integers with  $\displaystyle\liminf_{n\rightarrow\infty}\frac{a_{n+1}}{a_n}>2.$ Let $\varepsilon
    >0$ and $(b_n)_n$ be a non-decreasing sequence of positive integers with $(b_1b_2\cdots b_n)^{1+d+\epsilon}<b_{n+1}$ for sufficiently large values of $n.$ Then the number 
    $$
   \delta=\displaystyle\prod_{n=1}^{\infty}\frac{[b_n\alpha^{a_n}]}{b_n\alpha^{a_n}}
    $$
    is transcendental.
\end{theorem}
Our second result generalizes a earlier result of \cite{han}.
\begin{theorem}\label{th6.1}
    Let $\alpha>1$ be a real algebraic number of degree $d$ such that no power of $\alpha$ is a Pisot number.  
    Let $\delta$ and $\varepsilon$ be two positive real numbers with 
    $$
    \frac{1+d+\delta}{1+d}\cdot\frac{\varepsilon}{1+\varepsilon}>1.
    $$ 
    Suppose that $(a_n)$ and $(b_n)$ be two sequence of positive integers such that the sequence $(B_n)=(b_n\alpha^{a_n})$ is non-decreasing and 
    $$
\limsup_{n\to\infty}B_n^{1/(2+d+\delta)^n}=\infty.
$$ 
Assume that $B_n>n^{1+\epsilon}$ for sufficiently large $n.$ 
    Then the number 
    $$
   \delta=\displaystyle\prod_{n=1}^{\infty}\frac{[b_n\alpha^{a_n}]}{b_n\alpha^{a_n}}
    $$
    is transcendental.
\end{theorem}
\noindent\textbf{Remark 2.} 
In \cite[Theorem 5]{han}, the same transcendence result is proved with a restrictive condition on $\alpha$. Specifically, they assume that there exists a conjugate $\beta$ such that $\alpha\neq |\beta|>1$. In particular, no power of $\alpha$ is a Pisot number. Thus, Theorem \ref{th6.1} produces more transcendental numbers. For example, if we take $\alpha$ to be any Salem number, then no conjugate of $\alpha$ other than $\alpha$ has an absolute value strictly greater than 1. Therefore, we cannot apply \cite[Theorem 5]{han} to prove the transcendence of the infinite product $\displaystyle\prod_{n=1}^{\infty} \frac{[b_n\alpha^{a_n}]}{b_n\alpha^{a_n}}$. However, since no power of a Salem number is a Pisot number, by Theorem \ref{th6.1}, this infinite product represents a transcendental number. Furthermore, we can apply Theorem \ref{th6.1} to numbers that are not Salem numbers. For example, if $\alpha=\frac{1}{2}+\sqrt{2},$ then the infinite product $\displaystyle\prod_{n=1}^{\infty}[b_n\alpha^{a_n}]/b_n\alpha^{a_n}$ represents a transcendental number by Theorem \ref{th6.1}.
\newline


\begin{proof}[Proof of Theorem \ref{th6.2}]
    We prove by contradiction. Assume that $\delta$ is an algebraic number. Let $N_0$ be a sufficiently large positive integer. For $m\geq N_0,$ put $$p=p(m)=\prod_{n=1}^m[b_n\alpha^{a_n}]$$ and let $N=N(m)=\sum_{n=1}^ma_m.$ Then $$\left|\delta-\frac{p}{b_1\cdots b_m\alpha^N}\right|=\left|\frac{p}{b_1\cdots b_m\alpha^N}\right|\left|1-\prod_{n=m+1}^{\infty}\frac{[b_n\alpha^{a_n}]}{b_n\alpha^{a_n}}\right|;$$ using the inequality $|1-t|<|\log t|$ for $0<t<1,$ we deduce from the above that $$\left|1-\prod_{n=m+1}^{\infty}\frac{[b_n\alpha^{a_n}]}{b_n\alpha^{a_n}}\right|\leq \left|\log \left(\prod_{n=m+1}^\infty\frac{[b_n\alpha^{a_n}]}{b_n\alpha^{a_n}}\right)\right|. $$ On the other hand, $$\log \left(\prod_{n=m+1}^\infty\frac{[b_n\alpha^{a_n}]}{b_n\alpha^{a_n}}\right)=\sum_{n=m+1}^\infty\log \left(1-\frac{\left\{b_n\alpha^{a_n}\right\}}{b_n\alpha^{a_n}}\right),$$ where the symbol $\left\{\cdot\right\}$ stands for the fractional part. Using the inequality $|\log (1-t)|<|2t|$ for $0<t<1/2$ and the fact that the fractional part $\left\{\cdot\right\}$ is always less than $1,$ we find that the right hand side above is bounded by $$\sum_{n=m+1}^\infty\log \left(1-\frac{\left\{b_n\alpha^{a_n}\right\}}{b_n\alpha^{a_n}}\right)<\sum_{n=m+1}^\infty\frac{2}{b_n\alpha^{a_n}}.$$
    Since $b_n$ is a non-decreasing sequence, we obtain $$\sum_{n=m+1}^\infty\log \left(1-\frac{\left\{b_n\alpha^{a_n}\right\}}{b_n\alpha^{a_n}}\right)<\frac{2}{b_{m+1}\alpha^{a_m+1}}\sum_{n=m+1}^\infty\frac{1}{\alpha^{a_n-a_{m+1}}}<\frac{2}{b_{m+1}\alpha^{a_{m+1}}}\cdot\frac{1}{\alpha-1}.$$
    So finally using above inequalities and recalling that $p/b_1\cdots b_m\alpha^N\leq 1,$ we obtain
    \begin{equation}\label{eq6.5}
    \tag{6.1}
       \left|\delta-\frac{p}{b_1\cdots b_m\alpha^N}\right|<\frac{2}{b_{m+1}\alpha^{a_{m+1}}}\cdot\frac{1}{\alpha-1}. 
    \end{equation}
    
    Now since $\displaystyle\liminf_{n\rightarrow\infty}\frac{a_{n+1}}{a_n}>2,$ there exists $\varepsilon_1>0$ and $N_0\in\mathbb{N}$ such that for all $m\geq N_0$, $a_{m+1}>(2+\varepsilon_1)a_m.$ Thus for large $m,$ we have $a_{m+1}\geq (1+\varepsilon_1)N.$ Using this, (\ref{eq6.5}) and $(b_1b_2\ldots b_m)^{1+d+\varepsilon}<b_{m+1}$ for sufficiently large $m,$ we obtain 
$$      \left|\delta-\frac{p}{b_1\cdots b_m\alpha^N}\right|<\frac{2}{(b_1\cdots b_m)^{1+d+\varepsilon'}\alpha^{(1+\varepsilon')N}}\cdot\frac{1}{\alpha-1},$$ where $\varepsilon'=\mbox{min}\{\varepsilon_1,\varepsilon\}.$  Since no power of $\alpha$ is a Pisot number, $\left|\delta-\frac{p}{b_1\cdots b_m\alpha^N}\right|\neq 0.$ Now applying Corollary \ref{cor}, we deduce that some power of $\alpha$ is an algebraic number, which is a contradiction. Thus $\delta$ is a transcendental number.
\end{proof}

\begin{proof}[Proof of Theorem \ref{th6.1}]
Assume by contradiction that $\delta$ is an algebraic number. By following exactly as  the proof of \cite[Theorem 5]{han}, we obtain that there exists an $\varepsilon'>0$ such that for infinitely many $n\in\mathbb{N}$ $$0<\left|\delta-\frac{[b_n\alpha^{a_n}]}{\left(\prod_{k=1}^{n-1}b_k\right)\alpha^{\sum_{k=1}^{n-1}a_k}}\right|<\frac{1}{\left(\prod_{k=1}^{n}b_k\right)^{1+d+\varepsilon'} \alpha^{(1+\varepsilon')\sum_{k=1}^{n-1}a_k}}.$$ 
Now, by applying Corollary \ref{cor}, we conclude that some power of $\alpha$ is a Pisot number, which leads to a contradiction. Therefore, $\delta$ must be  a  transcendental number.
\end{proof}
\bigskip

\noindent{\bf Acknowledgments.} We are very  grateful to  Prof. Patrice Philippon and Prof. Dang Khoa Nguyen for their useful comments  in the preliminary version of this paper. We
are especially grateful to the referee for a careful reading and insightful comments about the earlier draft of this paper. The first author is thankful to NBHM-Research grant.


\end{document}